\documentclass[11pt]{article}
\usepackage{float, hyperref, graphicx, amsfonts, amsthm, amsmath, amssymb}
\theoremstyle{plain}
\newtheorem{theorem}{Theorem}
\newtheorem{lemma}[theorem]{Lemma}

\newtheorem{corollary}[theorem]{Corollary}

\theoremstyle{definition}

\theoremstyle{remark}

\author{Abbas Mehrabian \\ \small University of Waterloo \\ \texttt{amehrabi@uwaterloo.ca}}
\title{On the density of nearly regular graphs \\ with a good edge-labelling}
\date{}

\begin{document}
\maketitle

\begin{abstract}
A {good edge-labelling} of a simple graph
is a labelling of its edges with real numbers such that,
for any ordered pair of vertices $(u,v)$,
there is at most one nondecreasing path from $u$ to $v$.
Say a graph is {good} if it admits a good edge-labelling,
and is {bad} otherwise.
Our main result is that any good $n$-vertex graph whose maximum degree is within a constant factor of its average degree
(in particular, any good regular graph) has at most $n^{1+o(1)}$ edges.
As a corollary, we show that there are bad graphs with arbitrarily large girth,
answering a question of Bode, Farzad and Theis.
We also prove that for any $\Delta$, there is a $g$ such that
any graph with maximum degree at most $\Delta$ and girth at least $g$ is good.
\end{abstract}

\section{Introduction}
A \emph{good edge-labelling} of a simple graph
is a labelling of its edges with real numbers such that,
for any ordered pair of vertices $(u,v)$,
there is at most one nondecreasing path from $u$ to $v$.
This notion was introduced in~\cite{origin}
to solve wavelength assignment problems for specific categories of graphs.
Say graph $G$ is \emph{good} if it admits a good edge-labelling,
and is \emph{bad} otherwise.

Let $\gamma(n)$ be the maximum number of edges of a good graph on $n$ vertices.
Ara{\'u}jo, Cohen, Giroire, and Havet~\cite{gn} initiated the study of this function.
They observed that hypercube graphs are good,
and any graph containing $K_3$ or $K_{2,3}$ is bad, thus
$$\Omega(n\log n) \leq \gamma(n) \leq O(n\sqrt n).$$
Our main result is that
any good graph whose maximum degree is within a constant factor of its average degree
(in particular, any good regular graph)
has at most $n^{1+o(1)}$ edges.
Until now, no bad graphs with girth larger than 4 were known~\cite{gn,conjecture}.
Bode, Farzad and Theis~\cite{conjecture} asked whether all graphs with large enough girth are good.
As a corollary of our main result,
we give a negative answer by proving that there are bad graphs with arbitrarily large girth.
We also give a very short proof that the answer is positive for bounded degree graphs.

\section{The Proofs}
For a graph $G$ and an edge-labelling $\phi:E(G) \rightarrow \mathbb{R}$,
a \emph{nice $k$-walk from $v_0$ to $v_k$} is a sequence $v_0v_1\dots v_k$ of vertices
such that $v_{i-1}v_i$ is an edge for $1\leq i\leq k$, and
$v_{i-1}\neq v_{i+1}$ and $\phi(v_{i-1} v_{i}) \leq \phi (v_{i}v_{i+1})$ for $1 \leq i \leq k-1$.
The existence of a self intersecting nice walk
implies that the edge-labelling is not good:
let $v_0 v_1 \dots v_k$ be a shortest such walk with $v_0 = v_k$.
Then there are two nondecreasing paths $v_0 v_1 \dots v_{k-1}$ and $v_0 v_{k-1}$ from $v_0$ to $v_{k-1}$.
Thus if for some pair of vertices $(u,v)$ there are two nice $k$-walks from $u$ to $v$,
then the labelling is not good.

Let \emph{$f_k(n,m,\Delta)$} be the maximum number $f$
such that every edge-labelling of a graph on $n$ vertices,
at least $m$ edges and maximum degree at most $\Delta$,
has at least $f$ nice $k$-walks.

\begin{lemma}\label{lem:core}
Let $n,m,\Delta,k,a$ be positive integers
with $k>1$ and $a\leq \Delta/2$.
We have $f_1(n,m,\Delta) = m$ and
$$f_k(n,m,\Delta) \geq a \left[ f_{k-1}(n, m-an, \Delta-a) - (n\Delta - 2m) a (\Delta-a)^{k-3} \right].$$
\end{lemma}
\begin{proof}
Since any edge is a nice 1-walk, we have $f_1(n,m,\Delta) = m$.
Let $G$ be a graph with $n$ vertices, at least $m$ edges, and maximum degree at most $\Delta$.
Call a vertex of $G$ \emph{wealthy} if its degree is larger than $a$, and \emph{beggared} otherwise.
Let $b$ the number of beggared vertices.
Since every wealthy vertex has degree at most $\Delta$, and the sum of degrees is at least $2m$, we have
$$b a + (n-b) \Delta \geq 2m,$$
so $b \leq (n\Delta - 2m) / (\Delta - a)$.

Let $v$ be a wealthy vertex and $e_1,\dots,e_d$ be its incident edges,
ordered such that
$$\phi(e_1) \geq \phi(e_2) \geq \dots \geq \phi(e_d).$$
Call the edges $e_1,e_2,\dots,e_a$ the \emph{strong} edges for $v$.
Let $S$ be the set of all strong edges for all wealthy vertices.
Clearly $|S| \leq na$.
Let $H$ be the graph obtained from $G$ by deleting the edges in $S$.
Note that $H$ has $n$ vertices, at least $m-an$ edges,
and maximum degree at most $\max\{a, \Delta - a\} = \Delta - a$.

For a wealthy vertex $v$, every nice $(k-1)$-walk in $H$ ending in $v$
can be extended to $a$ distinct nice $k$-walks in $G$.
Thus every nice $(k-1)$-walk in $H$ whose both endpoints are wealthy,
can be extended to $a$ distinct nice $k$-walks in $G$.
By definition, there are at least $f_{k-1}(n, m-an, \Delta-a)$ nice $(k-1)$-walks in $H$.
The number of $(k-1)$-walks in $H$ starting from a beggared vertex
is not more than
$$b a (\Delta - a)^{k-2} \leq (n\Delta - 2m)a(\Delta-a)^{k-3},$$
since there are $b$ choices for the first vertex,
at most $a$ choices for the second vertex,
and at most $\Delta-a$ choices for the other $k-2$ vertices.
Hence there are at least
$$f_{k-1}(n, m-an, \Delta-a) - (n\Delta - 2m) a (\Delta-a)^{k-3}$$
nice $(k-1)$-walks in $H$ whose both endpoints are wealthy,
and the lemma follows.
\end{proof}

Let $q\in (0,1/2)$ be a fixed number that will be determined later, and let $p=1-q$.
Setting $a=q\Delta$ in the lemma gives
\begin{equation*}
\label{eq:core}
f_k(n,m,\Delta) \geq q\Delta f_{k-1}(n, m-qn\Delta, p\Delta) - q^2p^{k-3}\Delta^{k-1}(n\Delta-2m),
\end{equation*}
provided that $q\Delta$ is an integer.

Define two sequences $(a_i)_{i=1}^{\infty}$ and $(b_i)_{i=1}^{\infty}$
by $a_1 = 1, b_1 = 0$, and for $k>1$,
\begin{align*}
a_k = & qp^{k-2}a_{k-1} + 2 q^2 p^{k-3} \\
b_k = & q^2 p^{k-2} a_{k-1} + q p^{k-1} b_{k-1} + q^2 p^{k-3},
\end{align*}
And define the function $g_k(n,m,\Delta)$ as
$$g_k(n,m,\Delta) = a_k m \Delta^{k-1} - b_k n \Delta^k.$$
One computes $g_1(n,m,\Delta) = m$ and
$$g_k(n,m,\Delta)  =  q\Delta g_{k-1}(n, m-qn\Delta, p\Delta) - q^2p^{k-3}\Delta^{k-1}(n\Delta-2m).$$
Hence
$$f_1(n,m,\Delta) = g_1(n,m,\Delta),$$
and it is easy to show by induction on $k$ that given $t$,
\begin{equation}
f_k(n,m,\Delta) \geq g_k(n,m,\Delta),
\label{eq:geq}
\end{equation}
for $1\leq k \leq t$, provided that $q\Delta, qp\Delta, \dots, qp^{t-2}\Delta$ are positive integers.

\begin{lemma}
For any positive integers $t$ and $c$,
if $q$ is sufficiently small then $a_t > c b_t$.
\end{lemma}

\begin{proof}
Define $x_k = a_k/q^{k-1}$ and $y_k = b_k / q^{k-1}$. Then
\begin{align*}
x_1 = 1, y_1 = 0,\\
x_k = &  p^{k-2}x_{k-1} + 2 q   p^{k-3}, \\
y_k = & q   p^{k-2} x_{k-1} +   p^{k-1} y_{k-1} + q   p^{k-3}.
\end{align*}
Clearly, $a_t > cb_t$ if and only if $x_t > cy_t$.
Note that since $p=1-q<1$, we have $x_k \leq x_{k-1} + 2q$.
Assume that $q < 1/2t$.
So $x_k \leq 2$ for all $1\leq k\leq t$.

Now let $z_k = x_k - cy_k$. Then $z_1 = 1$ and
$$z_k = p^{k-2} \left(x_{k-1} - cp y_{k-1}\right) +qp^{k-3}(2-c) -cqp^{k-2}x_{k-1}.$$
Note that $p< 1$, $x_{k} \leq 2$ and $y_{k} \geq 0$ for all $1\leq k \leq t$,
so for $k$ in this range,
$$z_k \geq p^{k-2} z_{k-1} - 3cq.$$
Hence,
$$z_t \geq p^{t-2}p^{t-3}\dots p^2 p - 3cq(t-1) \geq  p^{t^2/2} - 3cqt.$$

Define $h(q) := (1-q)^{t^2/2} - 3cqt$.
Since $h(0) =1$ and $h$ is continuous,
there is a $q_0 > 0$ such that $h(q) > 0$ for all $0 \leq q < q_0$.
So for $0 < q < \min \{\frac{1}{2t}, q_0\}$ we have $a_t > c b_t$.
\end{proof}

Now we prove our main result,
which states that any good graph whose maximum degree is within a constant factor of its average degree
(in particular, any good regular graph) has average degree $n^{o(1)}$.
For a graph $G$, denote its maximum degree and average degree by $\Delta(G)$ and $\overline{d}(G)$, respectively.

\begin{theorem}\label{thm:main}
For any positive integers $t$ and $c$ there is an $\epsilon(t,c)>0$
such that any $n$-vertex graph $G$ with $\Delta(G) \leq c\overline{d}(G)$ and
$\epsilon(t,c) \overline{d}(G)^t > n$ is bad.
\end{theorem}

\begin{proof}
Let $q'$ be a large enough integer so that for $q=2^{-q'}$, $a_t - 4c b_t > 0$.
Let $q = 2^{-q'}$ and $\alpha_t = \frac{a_t}{4} - c b_t > 0$.
We claim that $\epsilon(t,c) = \min \{ c^{t-1}\alpha_t, 2^{-q' t^2}\}$ works.

Let $G$ be an $n$-vertex graph with $\Delta(G) \leq c\overline{d}(G)$ and
$\epsilon(t,c) \overline{d}(G)^t > n$.
Let $\overline{d} = \overline{d}(G)$ and
$r=2^{r'}$, where $r' = \lceil \log_2 \overline{d} \rceil$, so $r/2 < \overline{d} \leq r$.
We have
$$2^{-q' t^2} r^t \geq \epsilon(t,c) \overline{d}^t >  n \geq 1,$$
so $r > 2^{q't}$ and thus $qcr, qpcr, \dots, qp^{t-2}cr$ are positive integers.
Hence (\ref{eq:geq}) with $m=\frac{nr}{4}$ and $\Delta=cr$ holds for $1\leq k \leq t$ and thus
$$f_t\left(n,\frac{nr}{4},cr\right) \geq g_t\left(n,\frac{nr}{4},cr\right) = a_t \left(\frac{nr}{4}\right) (cr)^{t-1} - b_t n (cr)^t =
nr^tc^{t-1}\alpha_t \geq n\overline{d}^t\epsilon(t,c) > n^2.$$

Let $\phi$ be any edge-labelling of $G$.
Note that $G$ has at least $n r/4$ edges and maximum degree at most $cr$,
so $f_t(n,nr/4,cr)  > n^2$ means that $G$ has more than $n^2$ nice $t$-walks.
By the pigeonhole principle, there is an ordered pair of vertices $(u,v)$
such that there are two distinct nice $t$-walks from $u$ to $v$,
hence the labelling is not good.
\end{proof}

\begin{corollary}
For any integer $g \geq 3$ there is a bad graph with girth $g$.
\end{corollary}

\begin{proof}
Since $K_3$ and $K_{2,3}$ are bad, we may assume that $g \geq 5$.
Let $t$ be a positive integer larger than $3g/4$,
and let $d$ be an odd prime power larger than $2 / \epsilon(t,1)$.
Lazebnik, Ustimenko and Woldar~\cite{cages} proved that
there is a $d$-regular graph $G$ with girth $g$ with at most $2d^{\frac{3}{4} g - 1}$ vertices.
So
$$|V(G)| \leq 2 d^{\frac{3}{4} g - 1} < \epsilon(t,1) d^t,$$
and $G$ is bad by Theorem~\ref{thm:main}.
\end{proof}

Next we show that for any $\Delta$, there is a $g=g(\Delta)$ such that
any graph with maximum degree at most $\Delta$ and girth at least $g$ is good.

\begin{theorem}\label{thm:maxdeg}
Let $G$ be a graph with girth at least $2k$ and maximum degree at most $\Delta$.
If
$$4e k^2(\Delta-1)^{k-1} < k!$$
then $G$ admits a good edge-labelling.
\end{theorem}

\begin{proof}
Choose the label of each edge independently and uniformly at random from the interval $[0,1]$.
If the labelling is not good, then since the graph has girth at least $2k$,
there must exist a nondecreasing path of length exactly $k$.
For any path of length $k$, the probability that it is a nondecreasing path is $2/k!$.
Moreover, every path of length $k$ intersects at most $2k^2(\Delta-1)^{k-1} - 1$ other paths of length $k$.
Hence by the Lov{\'a}sz Local Lemma (see, e.g., Chapter 5 of \cite{alon}) there is a positive probability that the edge-labelling is good, and the proof is complete.
\end{proof}


\noindent\textbf{Acknowledgement.}
The author is grateful to Nick~Wormald for continuous support and helpful discussions.

\end{document}